\documentclass[a4paper]{amsart}

\usepackage{amsmath,amssymb,amsthm,amsfonts}
\usepackage[shortlabels]{enumitem}
\usepackage[english]{babel}
\usepackage{microtype}
\usepackage[colorlinks=false, pdfborder={0 0 0}]{hyperref}
\usepackage{amsmath}
\usepackage{amsthm}
\usepackage{amssymb}
\usepackage{amsfonts}
\usepackage{bm}
\usepackage[shortlabels]{enumitem}
\usepackage{marginnote} 
\usepackage{bbm}
\usepackage[all]{xy}
\usepackage{tikz}
\usetikzlibrary{arrows}
\usetikzlibrary{shapes,decorations}
\usetikzlibrary{positioning}
\usepackage{tikz-cd} 
\usepackage[normalem]{ulem}
\usepackage{mathrsfs} % for \mathscr

\newcommand{\norm}[1]{\left\lVert#1\right\rVert}

\DeclareMathOperator{\Lip}{Lip}

\newcommand{\dist}{\operatorname{dist}}

\newtheorem{theorem}{Theorem}[section]
\newtheorem{proposition}[theorem]{Proposition}
\newtheorem{corollary}[theorem]{Corollary}

\newtheorem{fact}[theorem]{Fact}
\newtheorem{question}[theorem]{Question}

\theoremstyle{definition}
\newtheorem{definition}[theorem]{Definition}
\theoremstyle{remark}
\newtheorem{remark}[theorem]{Remark}
\newtheorem{example}[theorem]{Example}

\def\<{\langle}
\def\>{\rangle}

\def\dU{d_\mathcal U}
\def\bx{\bar{x}}
\def\by{\bar{y}}
\def\bf{\bar{f}}

\def\U{\mathcal{U}}
\def\R{\mathbb{R}}

\def\to{\rightarrow}

\begin{document}

\title[Lipschitz-free spaces, ultraproducts, and finite representability]{Lipschitz-free spaces, ultraproducts, and finite representability of metric spaces}

\author[Luis C. Garc\'ia-Lirola]{Luis C. Garc\'ia-Lirola}
\address[L. Garc\'ia-Lirola]{Departamento de Matem\'aticas, Universidad de Zaragoza, 50009, Zaragoza, Spain}
\email{luiscarlos@unizar.es}

\author[G. Grelier]{ Guillaume Grelier }
\address[G. Grelier]{Universidad de Murcia, Departamento de Matem\'aticas, Campus de Espinardo 30100 Murcia, Spain} \email{g.grelier@um.es}



\maketitle

\begin{abstract} 
We study several properties and applications of the ultrapower $M_{\mathcal U}$ of a metric space $M$. We prove that the Lipschitz-free space $\mathcal F(M_{\mathcal U})$ is finitely representable in $\mathcal F(M)$. We also characterize the metric spaces that are finitely Lipschitz representable in a Banach space as those that biLipschitz embed into an ultrapower of the Banach space. Thanks to this link, we obtain that if $M$ is finitely Lipschitz representable in a Banach space $X$, then $\mathcal F(M)$ is finitely representable in $\mathcal F(X)$. We apply these results to the study of cotype in Lipschitz-free spaces and the stability of Lipschitz-free spaces and spaces of Lipschitz functions under ultraproducts. 
\end{abstract}

\section{Introduction}
Ultraproducts of Banach spaces are a very powerful tool to study local properties of Banach spaces as well as in the non-linear theory (see e.g. \cite{Heinrich, HeinrichMankiewicz, Sims}). In fact, a Banach  space $X$ is finitely representable in a Banach space $Y$ if and only if it is linearly isometric to a subspace of an ultrapower of $Y$. Many relevant notions have been characterized in terms of finite representability. For example, a Banach space $X$ has non-trivial Rademacher type (resp. cotype) if and only if $\ell_1$ (resp. $\ell_\infty)$ is not finitely representable in $X$ (see e.g. \cite{AK}). Moreover, the concept of super-reflexivity introduced by James in \cite{James} is another important example that can be characterised in terms of ultraproducts.

In this paper, we consider the notion of  ultraproduct of metric spaces (which is a generalization of the corresponding one for Banach spaces). We apply it to obtain an ultraproduct characterization of the metric spaces that are finitely Lipschitz representable (in the sense introduced by Lee, Naor and Peres \cite{LNP}) in a Banach space. Also, we analyze the relation between finite Lipschitz representability of metric spaces and finite representability of the corresponding Lipschitz-free spaces. These spaces (also called Arens-Eells spaces and transportation cost spaces) have become a very active research topic due to their applications in Non-Linear Analysis \cite{GK}, as well as Computer Science  and Optimal Transport. 

More precisely, given a metric space $M$ and an ultrafilter $\U$, 
we prove that the Lipschitz-free space on the ultrapower of $M$, $\mathcal F(M_{\U})$, is linearly isometric to a subspace of the ultrapower of the Lipschitz-free spaces, $\mathcal F(M)_{\U}$ (see Theorem \ref{th:isometry_lipfree}). In particular, $\mathcal F(M_{\U})$ is finitely representable in $\mathcal F(M)$. Also, we prove that a metric space $M$ is finitely Lipschitz representable into a Banach space $X$ if and only if $M$ biLipschitz embeds in an ultrapower of $X$ (Theorem~\ref{th:finrepBan}). As a consequence we obtain that, in such a case, $\mathcal F(M)$ is finitely representable in $\mathcal F(X)$. This result has some consequences on the cotype of Lipschitz-free spaces. For instance, the following dichotomy holds: either $\mathcal F(\ell_2)$ has non-trivial cotype or $\mathcal F(X)$ does not have cotype for any infinite-dimensional Banach space $X$. Finally, although several classes of Banach spaces (as Banach lattices, C*-algebras, and $C(K)$-spaces) are known to be stable by ultraproducts, we show that $(\mathcal F(M))_\U$ is not isomorphic to any Lipschitz-free space whenever $M$ is an infinite metric space and $\U$ is countably incomplete. We then compare the stability of $\Lip_0(K)$ and $\mathcal C(K)$ under ultraproducts and remark some similarities and differences between them.

The structure of the paper is the following. In the next section, we introduce the fundamental properties of ultraproducts of metric spaces and Lipschitz-free spaces used in this document. The main goal of Section \ref{sec:ultraLip} is to prove that if $M$ is a metric space and $M_\U$ is one of its ultrapowers then $\mathcal F(M_\U)$ is finitely representable in $\mathcal F(M)$. This will be an important tool in the rest of the paper. Section \ref{sec:finrep} links the ultraproducts of metric spaces with the notion of finite representability. Some applications to the cotype of free spaces are contained in Section \ref{sec:cotype}. Finally, in Section \ref{sec:stability} we analyze the stability of the class of Lipschitz-free spaces and spaces of Lipschitz functions under ultraproducts. 
Our notation is standard and follows textbooks such as \cite{AK, BST}.

\bigskip

\section{Notation and basic properties}\label{sec:notation}

\subsection{Ultraproduct of metric spaces} An excellent reference on this topic is a revised unpublished version of \cite{Sims}. Since that version might not be available to the reader, we have chosen to include here the necessary definitions and properties. Let $I$ be any infinite set. An ultrafilter on $I$ is called \emph{countably incomplete} (in short CI) if there exists a sequence $(I_n)_{n\in \mathbb N}$ of elements of $\mathcal U$ such that $I_{n+1}\subset I_n$ for every $n$ and $\bigcap_{n=1}^\infty I_n= \emptyset$. Note that every non-trivial ultrafilter on $\mathbb N$ is countably incomplete.

From now on, $\mathcal U$ will denote a nonprincipal ultrafilter on $I$. Let $\{(M_i, d_i)\}_{i\in I}$ be a family of metric spaces and fix a distinguished point $0_i\in M_i$ for every $i\in I$. Let us consider the set
\[ \ell_\infty (M_i) = \left\{ (x_i)_{i\in I}\in \prod_{i\in I} M_i : \sup_{i\in I} d_i(x_i,0_i)< \infty\right\}. \]
Notice that for every $(x_i)_{i\in I},(y_i)_{i\in I}\in \prod_{i\in I} M_i$ we have $\sup_{i\in I}d_i(x_i,y_i)<\infty$. Therefore, one can consider 
\[ d((x_i)_{i\in I},(y_i)_{i\in I}) := \lim_{\mathcal U,i} d_i(x_i, y_i). \]
It is clear that $d$ is a pseudometric on $\ell_\infty(M_i)$. We consider the equivalence relation given by $(x_i)_{i\in I}\sim (y_i)_{i\in I}$ if and only if $d((x_i)_{i\in I},(y_i)_{i\in I})=0$. We denote \[(M_i)_\U=\ell_\infty(M_i)/\sim\]
and $\pi\colon \ell_\infty(M_i)\to (M_i)_\U$ the canonical projection. Then the expression $$\dU(\bx,\by):=d((x_i)_{i\in I},(y_i)_{i\in I}),$$ for $\bx,\by\in (M_i)_\U$ and $\pi((x_i)_{i\in I})=\bx$, $\pi((y_i)_{i\in I})=\by$, defines a metric on $(M_i)_\U$. For simplicity, we usually omit $\pi$ and we write $(x_i)_\U=\pi((x_i)_{i\in I}$. The metric space $((M_i)_\U, \dU)$ is called the \emph{ultraproduct of the metric spaces} $(M_i)_{i\in I}$. Moreover, if $M_i=M$ and $0_i=0\in M$ for every $i\in I$ then the space $(M_i)_\U$ is called the \emph{ultrapower of the metric space} $M$ and denoted $M_\U$. If the context is clear, we simply write $d$ instead of $d_\U$.

Let us notice that if the spaces $M_i$ are uniformly bounded %(that is, $\sup\{d_i(x,y): x,y\in M_i, i\in I\}<\infty$),
then the definition of $(M_i)_\U$ does not depend on the choice of the distinguished points. In the case that the $M_i$ are normed spaces, we will always consider that the distinguished point is $0\in M_i$ for every $i\in I$ and we recover then the usual definition of ultraproduct for Banach spaces.

The following result summarises several known properties of the ultraproduct of metric spaces.  
We include the proofs (analogous to the Banach case ones) for completeness.

\begin{fact}\label{facts}
\begin{itemize}
\item[(a)] If $0_i\in N_i\subset M_i$ for each $i\in I$, then $(N_i)_\U$ embeds isometrically in $(M_i)_\U$. Moreover, if $\mathcal U$ is CI and $N_i$ is dense in $M_i$ for every $i\in I$ then $(N_i)_\U$ is isometric to $(M_i)_\U$. 
\item[(b)] If $\mathcal U$ is CI, then $(M_i)_\U$ is a complete metric space.
\item[(c)] If the $M_i$ are normed spaces, then $(M_i)_\U$ is a Banach space. 
\item[(d)] $M$ embeds isometrically in $M_\U$. Moreover, if $M$ is a normed space then there exists a linear isometry from $M$ into $M_\U$.
\item[(e)] If $M$ is a proper metric space (that is, closed balls in $M$ are compact sets) then $M_\U$ is isometric to $M$. 
\end{itemize}
\end{fact}

\begin{proof}
$(a)$ Consider $f_i\colon N_i\to M_i$ the canonical inclusion for each $i\in I$. Then it is straightforward that $(f_i)_{i\in I}$ defines an isometry from $(N_i)_\U$ into a subset of $(M_i)_\U$. In other to prove the second statement, take $\bx\in (M_i)_{i\in I}$ and fix $(x_i)_{i\in I}$ with $\bx=(x_i)_{i\in I}$. Take a decreasing sequence $(I_n)_{n\in \mathbb N}\subset \mathcal U$ such that $\bigcap_{n=1}^\infty I_n = \emptyset$. We will define $y_i\in N_i$ for each $i\in I$ satisfying that $\lim_{i,\U} d_i(x_i,y_i)=0$, so $\bx=(y_i)_{i\in I}$. If $i\notin I_1$, take $y_i\in M_i$ arbitrary. If $i\in I_n\setminus I_{n+1}$, take $y_i\in N_i$ so that $d_i(x_i,y_i)<1/n$. Since $\bigcap_{n=1}^\infty I_n = \emptyset$, this defines $y_i$ for every $i\in I$. Now notice that
\[ \{i\in I : d_i(x_i,y_i)<1/n\} \supset I_n \in \U \]
and so $\lim_{i,\U}d_i(x_i,y_i)=0$. This shows that $(f_i)_{i\in I}$ is onto, as desired. 

$(b)$ By the previous property, we may assume that $M_i$ is complete for each $i\in I$. Notice that $\pi\colon (\ell_\infty(M_i),d_\infty)\to ((M_i)_\U, \dU)$ is $1$-Lipschitz and onto. Since completeness is preserved by uniformly continuous surjections, we only need to check that $(\ell_\infty(M_i),d_\infty)$ is complete. For that, mimic the proof of the completeness of $\ell_\infty$. 

$(c)$ It is clear that $\dU$ is a norm on $(M_i)_\U$ whenever the $M_i$ are normed spaces. Moreover, $\ell_\infty(M_i)$ is a Banach space and $N_\U=\{(x_i)_{i\in I}\ |\ \lim_\U\|x_i\|=0\}$ is a closed subspace. So $(M_i)_\U=\ell_\infty(M_i)/N_\U$ is a Banach space.

$(d)$ Given $x\in M$, take $x_i=x$ for every $i\in I$. Then $(x_i)_{i\in I}\in \ell_\infty(M)$. Thus $\phi(x):= (x_i)_{i\in I}$ defines an isometry from $M$ into a subset of $M_\U$. For the last statement, notice that the map $\psi$ is a linear operator whenever $M$ is a normed space. 

$(e)$ Given $\bx\in M_\U$, take $(x_i)_{i\in I}$ such that $\bx=(x_i)_{i\in I}$. Then $R=\sup\{d(x_i,0):i\in I\}<\infty$ and so $\{x_i:i\in I\}$ is contained in the compact set $\overline{B}(0,R)$. Therefore, there exists $\psi(\bx):= \lim_{\mathcal U,i} x_i$. Notice that 
\[ d(\psi(\bx),\psi(\by)) = d\left(\lim_{\mathcal U,i} x_i, \lim_{\mathcal U,i} y_i\right) = \lim_{\mathcal U,i} d(x_i,y_i) = \dU(\bx,\by),\]
so $\psi$ defines an isometry from $M_\U$ into $M$. Moreover, given $x\in M$ we have $x=\psi(\phi(x))$ and therefore $\psi$ is onto.

\end{proof}

Note in passing that the ultraproduct of metric spaces is closely related to the Gromov-Hausdorff limit. Indeed, if $M$ is the Gromov-Hausdorff limit of a sequence of pointed proper metric spaces $M_n$, then $M$ is isometric to the ultraproduct $(M_n)_{\mathcal U}$ (see e.g. \cite{BridsonHaefliger}).

\subsection{Lipschitz-free spaces} Let $M$ be a pointed metric space $M$, that is, a metric space with a distinguished point denoted $0$. We will denote by $\Lip_0(M)$ the Banach space of all real-valued Lipschitz functions on $M$ vanishing at $0$, endowed with the norm given by the Lipschitz constant:
\[ \norm{f}=\Lip(f)=\sup\left\{\frac{|f(x)-f(y)|}{d(x,y)}: x,y\in M, x\neq y\right\}\]
For $x\in M$, the linear map $\delta(x) \colon \Lip_0(M) \to \R$ given by $\langle f,\delta(x) \rangle = f(x)$ defines an element of $\Lip_0(M)^*$. 
The \textit{Lipschitz-free space over $M$} (also called Arens-Eells space and transportation cost space) is defined as the closed subspace of $\Lip_0(M)^*$ generated by these evaluation functionals, that is,
\[\mathcal F(M) := \overline{ \operatorname{span}}^{\| \cdot  \|}\left \{ \delta(x) \, : \, x \in M  \right \} \subset \Lip_0(M)^*.\]
The map $\delta$ defines an isometric embedding of $M$ into $\mathcal F(M)$ such that the following fundamental property holds: for every Banach space $X$ and every Lipschitz function $f\colon M\to X$ with $f(0)=0$, there is a unique bounded linear operator $T_f\colon \mathcal F(M)\to X$ such that $T_f\circ \delta =f$ and $\norm{T_f}=\Lip(f)$. It follows in particular that $\mathcal F(M)^*=\Lip_0(M)$. We refer the reader to the monographs \cite{Ostrovskii, Weaver} and the survey \cite{GodefroySurvey} for more properties and applications of these spaces.

\bigskip

\section{Ultraproduct of $\mathcal F(M)$ and $\Lip_0(M)$}\label{sec:ultraLip}

Recall that, for a Banach space $X$, the ultrapower $(X^*)_{\U}$ embeds isometrically into $(X_{\U})^*$ and it is norming for $X_{\U}$. The following result provides an analogous for metric spaces (with Lipschitz functions playing the role of linear functionals). We just need to recall that, given $\lambda\geq 1$, a set $A\subset X^*$ is \textit{$\lambda$-norming for $X$} if $\sup_{x^*\in A\cap B_{X^*}}|x^*(x)|\geq\frac{1}{\lambda}\norm{x}$ for every $x\in X$.

\begin{theorem}\label{th:ultrLip} Let $\U$ be an ultrafilter on a set $I$ and let $(M_i)_{i\in I}$ be a family of metric spaces. Define an operator $T\colon\Lip_0(M_i)_\U\to\Lip_0((M_i)_\U)$ where $T((f_i)_\U)\in \Lip_0((M_i)_\U)$ is the function given by $$T((f_i)_\U)((x_i)_\U)=\lim_{\U,i} f_i(x_i).$$
Then  $T$ is a well-defined linear operator with  $\|T\|\leq 1$ and $T(B_{\Lip_0(M_i)_\U})$ is a 1-norming set for $\mathcal F((M_i)_\U)$.
\end{theorem}

\begin{proof} First notice that if 
$(x_i)_{i\in I}= (y_i)_{i\in I}$ in $(M_i)_{\U}$ and $(f_i)_{i\in I}=(g_i)_{i\in I}$ in $(\Lip_0(M_i))_{\mathcal U}$, then
\[ \lim_{\U,i}|f_i(x_i)-g_i(y_i)|\leq \lim_{\U,i}(\norm{f_i-g_i} d(x_i,0) + \norm{g_i} d(x_i,y_i)) = 0. \] So the formula does not depend on the chosen representations. Moreover, if $\bf=(f_i)_\U\in\Lip_0(M_i)_\U$ we have that
\[ |T\bf(\bx)-T\bf(\by)| \leq \norm{\bf} d(\bx,\by) \quad \forall \bx,\by\in (M_i)_{\mathcal U} \]
so $T\bf\in \Lip_0((M_i)_\U)$. Therefore $T$ is well-defined and we have that $\norm{T\bf}\leq \norm{\bf}$ for each $\bar{f}\in \Lip_0(M_i)_{\U}$.

Now, we will prove that $T(B_{\Lip_0(M_i)_{\U}})$ is a $1$-norming set. It is well known (see e.g.  Proposition~3.3 in  \cite{Kalton04}) that for that if suffices to check that given $\varepsilon>0$, a finite subset $A\subset (M_i)_\U$,  $\varepsilon>0$ and  $f\in\Lip_0((M_i)_\U)$, there exists $(f_i)_\mathcal U\in\Lip_0(M_i)_\U$ such that $\|(f_i)_\U\|\leq(1+\varepsilon)\|f\|$ and $T((f_i)_\U)|_{A}=f|_{A}$.

 Let $A=\{(x^j_i)_\U\}_{1\leq j\leq n}$ be a finite set, we may assume that  $0=(0_i)_\U\in A$ and that the $(x^j_i)_\U$ are all different. Fix $\varepsilon>0$ and $f\in\Lip_0((M_i)_\U)$. For $j,j'\in\{1,\ldots,n\}$ distinct, we have that $$I_{j,j'}=\left\{i\in I\ |\ d(x^j_i,x^{j'}_i)>\frac{1}{1+\varepsilon}d((x^j_i)_\U,(x^{j'}_i)_\U)\right\}\in\mathcal U.$$ It follows that $J=\bigcap_{j\neq j'}I_{j,j'}\in\mathcal U$ and then we can assume that $$d(x^j_i,x^{j'}_i)>\frac{1}{1+\varepsilon}d((x^j_i)_\U,(x^{j'}_i)_\U)$$  for all $j\neq j'$ and all $i\in I$. For $i\in I$, define a function $f_i\colon\{x^1_i,\ldots,x^n_i\}\to\mathbb R$ by $f_i(x^j_i)=f((x^j_k)_{\U,k})$ for all $j\in\{1,\ldots,n\}$, note that $f_i(0_i)=0$. If $j\neq j'$, we have that
\begin{align*}
|f_i(x^j_i)-f_i(x^{j'}_i)|&=|f((x^j_k)_{\U,k})-f((x^{j'}_k)_{\U,k})|\\
&\leq \|f\|d((x^j_k)_{\U,k},(x^{j'}_k)_{\U,k})\leq (1+\varepsilon)\|f\|d(x^j_i,x^{j'}_i)    
\end{align*}
proving that $f_i$ is $(1+\varepsilon)\|f\|$-Lipschitz and belongs to $\Lip_0(\{x^1_i,\ldots,x^n_i\})$. Now we extend $f_i$ to a $(1+\varepsilon)\|f\|$-Lipschitz function on $M_i$ and we still denote it by $f_i$. We have that $\|(f_i)_\mathcal U\|\leq (1+\varepsilon)\|f\|$ and $$T((f_i)_\mathcal U)((x^j_i)_\U)=\lim_{\U,i}f_i(x^j_i)=\lim_{\U,i}f((x^j_k)_{\U,k})=f((x^j_k)_{\U,k})$$ for all $j\in\{1,...,n\}$, proving that $T((f_i)_\mathcal U)$ and $f$ coincide on $A$. A standard argument using the denseness of finitely supported elements in $\mathcal F((M_i)_\U)$ gives that  $T(B_{\Lip_0(M_i)_{\U}})$ is 1-norming for $\mathcal F((M_i)_{\U})$, as desired.
\end{proof}

The operator $T$ defined in the previous theorem is not injective, in general. Indeed, we have the following characterization. 

\begin{proposition}\label{prop:unifdisc}
Let $M$ be a metric space and let $\U$ be a CI ultrafilter on a set $I$. Let $T\colon \Lip_0(M)_\U\to\Lip_0(M_\U)$ defined in Theorem~\ref{th:ultrLip}. The following assertions are equivalent:
\begin{enumerate}
    \item[(i)] $M$ is uniformly discrete and bounded;
    \item[(ii)] $T$ is injective;
    \item[(iii)] $T$ is an isometry.
\end{enumerate}
\end{proposition}

\begin{proof}
$(iii)\implies (ii)$ is obvious. 

$(ii)\implies (i)$ Let $(I_n)_n\subset \mathcal U$ be decreasing sequence of sets having empty intersection. Suppose by contradiction that $M$ is unbounded. Given $i\in I_n\setminus I_{n+1}$, consider the function $f_i$  given by $f_i(x)=d(x, B(0,n))$.  It is easy to check that $\norm{f_i}=1$. Let $\bar{f}=(f_i)_\U$. Then it follows $\norm{\bar{f}}=1$ and $T\bar{f}=0$, which is a contradiction. It follows that $M$ is bounded.

Now suppose that $M$ is not uniformly discrete. Then there exist two sequences $(x_n)_n$ and $(y_n)_n$ in $M$ such that $x_n\neq y_n$ for all $n\in\mathbb N$ and $d_n:=d(x_n,y_n)\to 0$. We have that $x_n\neq 0$ or $y_n\neq 0$, so by taking a subsequence if necessary we can suppose without loss of generality that $y_n\neq 0$ for all $n\in\mathbb N$. For all $n\in\mathbb N$, we define a $1$-Lipschitz function $g_n\colon M\to\mathbb [0,d_n]$ by $$g_n(x)=\max\left\{d_n-d(y_n,x),0\right\}$$ 
and let $h_n\in \Lip_0(M)$ given by $h_n(x)=g_n(x)-g_n(0)$.  It is clear that $\norm{h_n}$=1 and $\norm{h_n}_{\infty}\leq d_n$. Now let $i\in I$ and define $f_i=h_n$ where $n$ is such that $i\in I_n\setminus I_{n+1}$. Let $\bar{f}=(f_i)_{\U}\in\Lip_0(M)_\U$. We have that $\norm{\bar{f}}=\lim_{\U} \norm{f_i}=1$, so $\bar{f}\neq 0$. However we have that $T\bar{f}=0$ since $\norm{h_n}_{\infty}\leq d_n$ for $n\in\mathbb N$, which is again a contradiction. So $M$ is uniformly discrete.

$(i)\implies (iii)$ Let $\theta=\inf\{d(x,y): x,y\in M, x\neq y\}>0$. Take $\bf\in \Lip_0(M)_\U$ and $(f_i)_{i\in I}$ with $\bf=(f_i)_\U$. By Theorem \ref{th:ultrLip}, we just need to prove that $\|T\bf\|\geq\norm{\bf}$. Let $\varepsilon>0$. For all $i\in I$, pick $x_i,y_i\in M$ two distinct points of $M$ such that $|f_i(x_i)-f_i(y_i)|\geq (1-\varepsilon)\norm{f_i}d(x_i,y_i)$. Since $M$ is bounded, we may consider $\bx=(x_i)_\U$ and $\by=(y_i)_\U$. Moreover, we have that $d(\bx,\by)=\lim_\U d(x_i,y_i)\geq\theta$, so $\bx\neq\by$. Taking limit on $\U$, it follows that $|T\bf(\bx)-T\bf(\by)|\geq(1-\varepsilon)\norm{\bf}d(\bx,\by)$ and then $\|T\bf\|\geq(1-\varepsilon)\norm{\bf}$. Since this is true for all $\varepsilon>0$, we obtain that $\|T\bf\|\geq\norm{\bf}$.
\end{proof}

Note that the implication $(i)\Rightarrow (iii)$ works for any ultrafilter (not necessarily CI).

\begin{remark} In general the operator $T$ is not onto. In fact, let $M$ be a bounded infinite uniformly discrete set. Suppose also that $\U$ is CI and let $(I_n)_n$ be a decreasing sequence in $\U$ with empty intersection. Let  $f=T((f_i)_\U)$.  Given $i\in I_n\setminus I_{n+1}$, take $x_i,y_i\in M$ two distinct points with $f_i(x_i)-f_i(y_i)\geq (1-1/n)\norm{f_i}d(x_i,y_i)$. Let $\bx=(x_i)_\U$ and $\by=(y_i)_\U$ in $M_\U$ and note that these two elements are distinct since $M$ is uniformly discrete. Then clearly $\bf(\bx)-\bf(\by)=d(\bx,\by)$, that is, 
\[ T((\Lip_0(M))_{\U}) \subset \operatorname{SNA}(M_{\U})\]
where $\operatorname{SNA}(M_{\U})$ denotes the set of Lipschitz functions on $N$ attaining their Lipschitz constant at a pair of points of $M_{\U}$. However, whenever the metric space $M$ is infinite, there are Lipschitz functions on $M_{\U}$ which do not attain the Lipschitz constant (otherwise, every linear functional on $\mathcal F(M_{\U})$ attains its norm, and then $\mathcal F(M_{\U})$ is reflexive). 
\end{remark}

\begin{theorem}\label{th:isometry_lipfree}
Let $\U$ be an ultrafilter on a set $I$ and let $(M_i)_{i\in I}$ be a family of metric spaces. Then $\mathcal F((M_i)_\U)$ is linearly isometric to $\overline{\operatorname{span}(\delta(M_i)_\U)}\subset\mathcal F(M_i)_\U$. 
\end{theorem}

\begin{proof} 
Let $s\colon (M_i)_\U\to\mathcal F(M_i)_\U$ defined by $s((x_i)_\U)=(\delta_{x_i})_\U$. Note that $s$ is an isometry since 
\begin{align*}
d((x_i)_\U,(y_i)_\U)&=\lim_\U d(x_i,y_i)=\lim_\U\|\delta_{x_i}-\delta_{y_i}\| 
=\|(\delta_{x_i})_\U-(\delta_{x_i})_\U\|\\ &=\|s((x_i)_\U)-s((y_i)_\U)\|
\end{align*}
for all $(x_i)_\U,(x_i)_\U\in (M_i)_\U$. By the linearization property of Lipschitz-free spaces, $s$ extends to a continuous linear operator $S\colon\mathcal F((M_i)_\U)\to\mathcal F(M_i)_\U$ such that $\|S\|=1$. 
Let $\varepsilon>0$ and fix $\mu=\sum_{j=1}^na_j\delta_{(x^j_i)_\U}\in \mathcal F((M_i)_\U)$. Let $T\colon\Lip_0(M_i)_\U\to\Lip_0((M_i)_\U)$ be the operator defined in Theorem~\ref{th:ultrLip}. Since $T(B_{\Lip_0(M_i)_\U})$ is $1$-norming, there exists $(f_i)_\mathcal U\in\Lip_0(M_i)_\U$ such that $\|\mu\|=\langle T((f_i)_\mathcal U),\mu\rangle$ and $\|(f_i)_\mathcal U\|\leq 1+\varepsilon$. It follows that
\begin{align*}
\|\mu\|&=\langle T((f_i)_\mathcal U),\mu\rangle
=\sum_{j=1}^na_j\langle T((f_i)_\mathcal U),\delta_{(x^j_i)_\U}\rangle
=\sum_{j=1}^na_j\lim_{\mathcal U,i} f_i(x_i^j)\\
&=\sum_{j=1}^na_j\langle (f_i)_\mathcal U, (\delta_{x^j_i})_\U\rangle
=\langle(f_i)_\mathcal U,S(\mu)\rangle \leq(1+\varepsilon)\|S(\mu)\|,
\end{align*}
and we deduce that $\|\mu\|\leq\|S(\mu)\|$ since $\varepsilon$ was arbitrary. By density of the measures with finite support, it follows that $S$ is an isometry.

\end{proof}

\begin{remark}
The previous proof gives that $\langle T((f_i)_\mathcal U),\mu\rangle=\langle(f_i)_\mathcal U,S(\mu)\rangle$ for all $\mu\in \mathcal F((M_i)_\U)$ and all $(f_i)_\mathcal U\in\Lip_0(M_i)_\U$. In other words,  $S^*|_{\Lip_0(M_i)_\U}=T$.
\end{remark}

\bigskip

\section{Finite representability of metric and Banach spaces}\label{sec:finrep}

Given $\lambda\geq 1$, a Banach space $X$ is \textit{$\lambda$-finitely representable} in a Banach space $Y$ if for any finite-dimensional subspace $E$ of $X$ and every $\varepsilon>0$, there exists a finite-dimensional subspace $F$ of $Y$ such that $d(E.F)\leq\lambda+\varepsilon$, where $d(E,F)$ is the Banach-Mazur distance between $E$ and $F$. If there exists $\lambda\geq 1$ such that $X$ is $\lambda$-finitely representable $Y$, we say that $X$ is \textit{crudedly finitely representable} in $Y$. Moreover, if $X$ is $1$-finitely representable in $Y$, we say that $X$ is \textit{finitely representable} in $Y$ (see e.g. \cite{AK} for these notions). It is well known  that $X$ is finitely representable in $Y$ if and only if $X$ is isometric to a subspace of an ultrapower of $Y$. Thus, we immediately obtain the following consequence of Theorem~\ref{th:isometry_lipfree}.

\begin{theorem}\label{th:norming}
Let $M$ be a metric space and let $\mathcal U$ be an ultrafilter. Then $\mathcal F(M_\U)$ is finitely representable in $\mathcal F(M)$.
\end{theorem}

We will deal with a related notion for metric spaces introduced by Lee, Naor and Peres in \cite{LNP}. We take the terminology from \cite{Baudier}. For a biLipschitz embedding $\phi$, $\dist(\phi)=\Lip(f)\Lip(f^{-1})$ denotes its \emph{distortion}.
\begin{definition}
Let $\lambda\geq 1$ and $M, N$ be metric spaces. We say that $M$ is \emph{finitely $\lambda$-Lipschitz representable} into $N$ if for every finite subset $F$ in $M$ and every $\varepsilon>0$ there is a map
$\phi\colon F \to N$ such that $\dist(\phi)\leq \lambda+\varepsilon$.
\end{definition}

Moreover, we will consider the following notions. 

\begin{definition} Let $M$ and $N$ be metric spaces. If $M$ is finitely $\lambda$-Lipschitz representable in $N$ for some $\lambda\geq 1$, we say that $M$ is \emph{crudely finitely Lipschitz representable} in $N$. 
If $M$ is finitely $1$-Lipschitz representable in $N$, we say that $M$ is \emph{finitely representable} in $N$.
\end{definition}

In the case of Banach spaces, this notion coincides with the usual finite representability. Indeed, the following is a consequence of Theorem 13 in \cite{Tony}.

\begin{proposition}
Let $X$ and $Y$ be two Banach spaces and let $\lambda\geq 1$. Then $X$ is finitely $\lambda$-Lipschitz representable in $Y$ if and only if $X$ is $\lambda$-finitely representable in $Y$.
\end{proposition}

Our first goal is to show that the finite Lipschitz representability admits a characterization in terms of ultraproducts which is analogous to the corresponding result for Banach spaces (see \cite{Heinrich}). 

\begin{proposition}\label{prop:LNPulta} Let $M, N$ be metric spaces. The following assertions are equivalent:
\begin{enumerate}[(i)]
\item $M$ is finitely $\lambda$-Lipschitz representable into $N$;
\item there exist an ultrafilter  $\mathcal U$ on a set $I$, scaling factors $r_i>0$, points $0_i\in N$ and a $\lambda$-biLipschitz embedding of $M$ into $(N, 0_i, r_i d)_{\U}$.
\end{enumerate}
In that case, if moreover $M$ is separable, then for any CI ultrafilter $\U$ there are $r_i>0$, points $0_i\in N$, and a $\lambda$-biLipschitz embedding of $M$ into $(N, 0_i, r_i d)_{\U}$.
\end{proposition}

\begin{proof} 
Suppose that $(i)$ holds. Fix a point $0\in M$ and define $$I:=\{(A,\varepsilon)\ :\ 0\in A\subset M,\ |A|<\infty,\ 0<\varepsilon<1\}$$ with the partial order defined by $(A_1,\varepsilon_1)\preceq(A_2,\varepsilon_2)$ if and only if $A_1\subset A_2$ and $\varepsilon_1\geq\varepsilon_2$. Since any pair of element of $I$ has an infimum, it is easy to show that $$\beta:=\{\{i\in I:\ i_0\preceq i\}\ |\ i_0\in I\}$$ is a filter basis. Then let $\U$ be any ultrafilter  containing the filter generated by $\beta$. For all $i=(A_i,\varepsilon_i)\in I$, there exists a one-to-one function $\phi_i\colon A_i\to N$ such that $\dist(\phi_i)\leq\lambda+\varepsilon_i$. Consider the metric space $(N_i,0_i,d_i)$ where $N_i=N$, $d_i= \norm{\phi_i^{-1}}d$ and $0_i=\phi_i(0)$. Given $x\in M$, let $y_i=\phi_i(x)$ if $x\in A_i$ with $i=(A_i,\varepsilon_i)$ and $y_i=0_i$ if not. Note that
\[ d_i(\phi_i(x),0_i)\leq \norm{\phi_i}d_i(x,0)\leq (\lambda+1)d(x,0)\]
and so $(\phi_i(x))_{i\in I}$ gives an element of $(N_i)_{\mathcal U}$. This means that $x\mapsto (y_i)_{\U}$ defines a map $\phi\colon M\to (N_i)_\U$. 

Now, let $\varepsilon_0>0$ arbitrary and take $x,x'\in M$. Note that $I_0:=\{(A,\varepsilon)\in I\ |\ x,x'\in A,\ \varepsilon\leq\varepsilon_0\}$ belongs to $\U$. For $i\in I_0$, we have that \[d(x,x')\leq \norm{\phi_i^{-1}} d(y_i,y_i')\leq \norm{\phi_i^{-1}}\norm{\phi_i} d(x,x')\leq (\lambda+\varepsilon_0)d(x,x').\] Letting $r_i=\norm{\phi_i^{-1}}$ and taking limit on $\U$, we obtain that \[d(x,y)\leq d(\phi(x),\phi(y))\leq (\lambda+\varepsilon_0) d(x,y).\] Since $\varepsilon_0$ was arbitrary, we conclude that $\phi$ is a $\lambda$-biLipschitz embedding.

For the other implication, suppose that there exists $\phi\colon M\to (N, 0_i, r_i d)_{\mathcal U}$ with $\dist(\phi)\leq \lambda$ for some ultrafilter $\U$ on a set $I$ and numbers $r_i>0$. Let $A=\{x^1,\ldots,x^p\}$ be a finite subset of different elements of $M$ and fix $\varepsilon>0$. Each $\phi(x^k)$ can be written $\phi(x^k)=(y^k_i)_\mathcal U$. For $i\in I$, define a function $\phi_i\colon A\to N$ by $\phi_i(x^k)=y^k_i$. Note that for $k,l\in\{1,\ldots,p\}$, we have that
\[ \norm{\phi^{-1}}^{-1}d(x^k,x^l)\leq d(\phi(x^k),\phi(x^l))\leq \norm{\phi}d(x^k,x^l)\]
and
\[ d(\phi(x^k),\phi(x^l))=\lim_{i,\U} r_i d(\phi_i(x^k),\phi_i(x^l)).\] 
It follows that 
\[\left\{i\in I\ |\ (1-\varepsilon)\norm{\phi^{-1}}^{-1}d(x^k,x^l)\leq r_id(\phi_i(x^k),\phi_i(x^l))\leq (1+\varepsilon)\norm{\phi}d(x^k,x^l)\ \forall k,l\right\}\]
belongs to $\mathcal U$ and so it is not empty. Taking $i$ in this set we have that \[(1-\varepsilon)r_i^{-1}\norm{\phi^{-1}}d(a,b)\leq d(\phi_i(a),\phi_i(b))\leq (1+\varepsilon)r_i^{-1}\norm{\phi}d(a,b)\] for all $a,b\in A$. That is, \[\dist(\phi_i)\leq \frac{1+\varepsilon}{1-\varepsilon}\dist(\phi)\leq \frac{1+\varepsilon}{1-\varepsilon}\lambda\]
and so $(i)$ holds.

Now suppose that $M$ is separable and that $(i)$ holds. Let $\U$ be any CI ultrafiltrer over a set $I$ and let $(I_n)_n\subset\U$ be a decreasing sequence with empty intersection. Let $\{x_n\}_n$ be a countable dense subset of $M$. For all $n\in\mathbb N$, there exists a function $\phi_n\colon \{x_k\}_{1\leq k\leq n}\to N$ such that $\dist(\phi_n)\leq (1+1/n)\lambda$. Given $i\in \bigcup_n I_n$, let $n_i$ be such that $i\in I_{n_i}\setminus I_{n_i+1}$, and consider the metric space $(N_i, 0_i, d_i)$ where $N_i=N$, $d_i=\norm{\phi_{n_i}^{-1}}d$ and $0_i=\phi_{n_i}(x_1)$. If $i\in I\setminus I_1$, define $r_i>0$ arbitrarily. Note that, given $m\in \mathbb N$, 
\[d_i(\phi_{n_i}(x_m),0_i)\leq \dist(\phi_{n_i})d(x_m,x_1)\leq 2\lambda d(x_m,x_1),\]
and so we may consider the element $(\phi_{n_i}(x_m))_{\U,i}$.

Now, define a function $\phi\colon\{x_n\}_n\to (N_i)_{\U}$ by $\phi(x_m)=(\phi_{n_i}(x_m))_{\mathcal U,i}$. We will prove that $\phi$ is an isometry and then will extend to a unique isometry defined on $M$. Let $\varepsilon>0$ and $p_0\in\mathbb N$ such that $\frac{1}{p_0}<\varepsilon$. Let $p<q$ and define $\tilde{q}=\max\{p_0,q\}$. Let $I_0=\bigcup_{n\geq\tilde{q}}I_n\in\U$ and take $i\in I_0$. It is clear that $y^p_i=\phi_{n_i}(x_p)$ and $y^q_i=\phi_{n_i}(x_q)$. It follows that 
\begin{align*} d(x_p, x_q)&\leq \norm{\phi_{n_i}^{-1}}d(y_i^p,y_i^q)\leq \norm{\phi_{n_i}}\norm{\phi_{n_i}^{-1}}d(x_p,x_q)\\
&\leq (1+1/n_i)\lambda d(x_p,x_q)<(1+\varepsilon)\lambda d(x_p,x_q)\end{align*}

Taking limit on $\U$, we deduce that $$d(x_p,x_q)\leq d(\phi(x_p),\phi(x_q))\leq (1+\varepsilon)\lambda d(x_p,x_q).$$ Since $\varepsilon$ was arbitrary, we conclude that $\phi$ is an isometry and the proof is complete.
\end{proof}

Note that in the case $N=X$ is a Banach space clearly one may assume that $\phi_i(0)=0$ and $\norm{\phi_i^{-1}}=1$ (and then $r_i=1$) in the proof of $(i)\Rightarrow(ii)$ above, so we get: 

\begin{theorem}\label{th:finrepBan}
Let $M$ be a metric space and $X$ be a Banach space. The following assertions are equivalent:
\begin{enumerate}
    \item[(i)] $M$ is finitely $\lambda$-Lipschitz representable in $X$;
    \item[(ii)] there exists an ultrafilter $\mathcal U$ such that $M$ is $\lambda$-biLipschitz equivalent to a subset of $X_\U$.
\end{enumerate}
In that case, if moreover $M$ is separable and $\mathcal U$ is a CI ultrafilter, then $M$ is $\lambda$-biLipschitz equivalent to a subset of $X_\U$.
\end{theorem}

\begin{theorem}\label{fr_metricspace}
Let $M$ be a metric space and $X$ be a Banach space. Assume  that $M$ is finitely $\lambda$-Lipschitz representable in $X$. Then $\mathcal F(M)$ is $\lambda$-finitely representable in $\mathcal F(X)$. 
\end{theorem}

\begin{proof}
Assume $M$ is finitely $\lambda$-Lipschitz representable in $X$. By Theorem~\ref{th:finrepBan}, there exists an ultrafilter $\U$ such that $M$ $\lambda$-biLipschitz embeds in  $X_\U$. It follows that $\mathcal F(M)$ is $\lambda$-isomorphic to a subspace of $\mathcal F(X_\U)$. By Theorem \ref{th:isometry_lipfree}, we deduce that $\mathcal F(M)$ is $\lambda$-isomorphic to a subspace of $\mathcal F(X)_\U$. This means exactly that $\mathcal F(M)$ is $\lambda$-finitely representable in $\mathcal F(X)$.
\end{proof}

\begin{remark} Note that if $M$ and $N$ are bounded metric spaces satisfying that for every finite subset $F\subset M$ and every $\varepsilon>0$ there exists a function $f\colon F\to N$ such that 
\[(1+\varepsilon)^{-1} d(x,y)\leq d(\phi(x),\phi(y))\leq (1+\varepsilon)d(x,y) \quad \forall x,y\in F,\]
then a similar argument shows that $\mathcal F(M)$ is finitely representable in $\mathcal F(N)$. 
\end{remark}

We obtain some immediate consequences:

\begin{corollary} \label{cor:fr} Let $X$ and $Y$ be Banach spaces. Then $\mathcal F(X)$ is finitely representable in $\mathcal F(Y)$ in any of the following cases:
\begin{enumerate}[(a)]
    \item $X=\ell_2$ and $Y$ is any infinite-dimensional Banach space.
    \item $X=Y^{**}$ and $Y$ is any Banach space.
    \item $X=L_p([0,1])$ and $Y=\ell_p$, where $1\leq p<\infty$.
\end{enumerate}
\end{corollary}
\begin{proof}
In each of the cases, we have that $X$ is finitely representable in $Y$. In fact, for (a) it is a consequence of Dvoretzky's theorem (see Theorem 6.15 in \cite{BST}). For (b), it is the principle of local reflexivity (see Theorem 6.3 in \cite{BST}) and (c) is part of Theorem 6.2 in \cite{BST}.
\end{proof}

\begin{corollary}
Let $X$ and $Y$ be Banach spaces such that $X$ coarsely Lipschitz embeds into $Y$. Then $\mathcal F(X)$ is crudely finitely representable in $\mathcal F(Y)$.
\end{corollary}

\begin{proof}
That follows from Ribe's theorem (see Theorem 14.2.27 in \cite{AK}).
\end{proof}

\bigskip

\section{Some remarks on the cotype of Lipschitz-free spaces}\label{sec:cotype}

Not much is known about the Rademacher cotype of Lipschitz-free spaces. Bourgain proved (\cite{Bourgain}, see also Theorem~10.16 in \cite{Ostrovskii}) that $\mathcal F(\ell_1)$ has trivial cotype, but whether $\mathcal F(\mathbb R^n)$ has a nontrivial cotype is a long-standing open problem. Note that as a consequence of Corollary~\ref{cor:fr} the following dichotomy holds: 
\begin{enumerate}
    \item[(a)] $\mathcal F(\ell_2)$ has cotype; or
    \item[(b)] $\mathcal F(X)$ does not have cotype for any infinite-dimensional Banach space $X$.
\end{enumerate}

We obtain now some remarks concerning the cotype of $\mathcal F(M)$. Recall that 
the notion of metric cotype was introduced by Mendel and Naor in \cite{MN}. Note that if $M$ is a metric space such that $\mathcal F(M)$ has Rademacher cotype $q$, then $M$ also has metric cotype $q$. In particular, if $M=X$ is a Banach space then $X$ has Rademacher cotype $q$ (this follows directly from the fact that the metric cotype passes to subspaces and is equivalent to the usual cotype for Banach spaces). 

On the other hand, the cotype of $\mathcal F(M)$ is related to the metric type introduced by Bourgain, Milman and Wolfson in \cite{BMW}.

\begin{proposition}
Let $M$ be a metric space such that $\mathcal F(M)$ has Rademacher cotype. Then $M$ has BMW type. In particular, if $M=X$ is a Banach space then $X$ has Rademacher type.
\end{proposition}

\begin{proof}
Suppose that $M$ does not have BMW type. By Theorem 2.6 in \cite{BMW}, $M$ contains uniformly biLipschitz copies of the Hamming cubes $\mathbb F_2^n$. Bourgain's result mentioned earlier provides a constant $C\geq 1$ such that for all $m$ there exists $n$ such that $\mathcal F(\mathbb F_2^n)$ contains a $C$-isomorphic copy of $\ell^m_\infty$. Since the space $\mathcal F(M)$ contains $D$-isomorphic copies of the spaces $\mathcal F(\mathbb F_2^n)$ for some $D\geq 1$, it follows that $\mathcal F(M)$ contains $CD$-isomorphic copies of the spaces $\ell^m_\infty$. In particular, $\mathcal F(M)$ can not have cotype. If $M$ is Banach space then $M$ has BMW type if and only if $M$ has Rademacher type by Corollary~5.9 in \cite{BMW}.
\end{proof}

\begin{remark}
If $X$ is a Banach space such that $\mathcal F(X)$ has Rademacher cotype, then we can deduce easily from Theorem \ref{fr_metricspace} that $X$ has Rademacher type. In fact, if $X$ does not have Rademacher type then $\ell_1$ is finitely representable in $X$ and then $\mathcal F(\ell_1)$ is finitely representable in $\mathcal F(X)$. This is a contradiction since $\mathcal F(\ell_1)$ does not have Rademacher cotype.
\end{remark}

It is not known which metric spaces $M$ satisfy that $\mathcal F(M)$ and $\mathcal F(\mathcal F(M))$ are isomorphic (one example is Pe\l czy\'nski universal space, see \cite{GK}). The next result shows in particular that if $\mathcal F(M)$ has cotype then $\mathcal F(M)$ and $\mathcal F(\mathcal F(M))$ cannot be isomorphic.

\begin{corollary}
Let $M$ be an infinite metric space. Then $\mathcal F(\mathcal F(M))$ does not have Rademacher cotype.
\end{corollary}

\begin{proof}
Suppose that $\mathcal F(\mathcal F(M))$ has cotype. It follows from the previous result that $\mathcal F(M)$ has type. This is impossible since $\mathcal F(M)$ contains an isomorphic copy of $\ell_1$.
\end{proof}

Aliaga, No\^us, Petitjean and Proch\'azka have proved recently in \cite{ANPP} that several isomorphic properties of $\mathcal F(X)$ (such as the Schur property and weak sequential completeness) are compactly determined. We finish the section by showing that this is also the case of the cotype. The proof adapts some ideas from \cite{Petitjean}. 

\begin{proposition}
Let $X$ be a Banach space and let $q\geq 2$. The following assertions are equivalent:
\begin{enumerate}
    \item[(i)] $\mathcal F(X)$ has Rademacher cotype (resp. cotype $q$);
    \item[(ii)] $\mathcal F(K)$ has Rademacher cotype (resp. cotype $q$)  for any (countable) compact set $K\subset X$;
    \item[(iii)] $\mathcal F(\{x_n\}_n)$ has Rademacher cotype (resp. cotype $q$) for any null sequence $(x_n)_n\subset X$.
\end{enumerate}
\end{proposition}

\begin{proof}
The implications $(i)\implies(ii)\implies(iii)$ are trivial. Suppose that $\mathcal F(X)$ does not have Rademacher cotype (resp. cotype $q$). It follows that $\mathcal F(2^{-n}B_X)$ does not have cotype (resp. cotype $q$) for all $n\in\mathbb N$. In particular, for all $n\in\mathbb N$, there exists $m\in \mathbb N$ and $\mu_1^n,\ldots,\mu_m^n\in\mathcal F(2^{-n}B_X)$ such that
$$\left(\sum_{k=1}^m\|\mu_k^n\|^n\right)^{\frac{1}{n}}>n\int_{0}^{1}\left\|\sum_{k=1}^m\mu_k^nr_k(t)\right\|\,dt$$
$$\left(\text{resp.} \ \left(\sum_{k=1}^m\|\mu_k^n\|^q\right)^{\frac{1}{q}}>n\int_{0}^{1}\left\|\sum_{k=1}^m\mu_k^nr_k(t)\right\|\,dt\right).$$ Since the measures with finite support are dense in a Lipschitz-free space, we can and do suppose that $\mu_1^n,\ldots,\mu_m^n\in\mathcal F(K_n)$ where $K_n$ is a finite subset of $2^{-n}B_X$. Define $K=\bigcup_nK_n\cup\{0\}$. Then $K$ is a null sequence such that $\mathcal F(K)$ does not have Rademacher cotype (resp. cotype $q$).
\end{proof}

\begin{remark}
Since $\mathcal F(\ell_1)$ does not have Rademacher cotype, the previous theorem implies that there exists a null sequence $(x_n)_n$ in $\ell_1$ such that $\mathcal F(\{x_n\}_n)$ does not have cotype. Moreover, it is possible to explicite such a sequence. For $n\geq 1$, define $$x_n=\frac{1}{k^2}\left(r_1\left(\frac{m}{2^k}\right),...,r_k\left(\frac{m}{2^k}\right),0,0,...\right)$$ where $k$ and $m$ are such that $2^k-1\leq n< 2^{k+1}-1$ and $n=2^k-1+m$ with $0\leq m\leq 2^k-1$. Note that $\mathcal F(\mathbb F^k_2)=\mathcal F(\{x_n\}_{2^k-1\leq n< 2^{k+1}-1})$ isometrically since the metric space $\{x_n\}_{2^k-1\leq n< 2^{k+1}-1}$ is obtained by scaling the distance on $\mathbb F^k_2$. It follows that $\mathcal F(\mathbb F^k_2)$ is an isometric subspace of $\mathcal F(\{x_n\}_n)$ for all $k\geq 1$. So $\mathcal F(\{x_n\}_n)$ does not have cotype.
\end{remark}

\bigskip

\section{Stability of $\mathcal F(M)$ and $\Lip_0(M)$ under ultraproducts}\label{sec:stability}

Several classes of Banach spaces, as Banach lattices, C*-algebras and $C(K)$ spaces, are stable under ultraproducts \cite{Heinrich}. Given a metric space $M$ and an  ultrafilter $\U$, it is natural to ask if $\mathcal F(M)_{\U}$ is isomorphic to $\mathcal F(M_{\U})$ or more generally if there exists a metric space $N$ such that $\mathcal F(M)_{\U}$ is isomorphic to $\mathcal F(N)$. The first question is easily seen to be false with the following example:

\begin{example}
Let $M$ be an infinite proper metric space. Then $M_{\U}=M$ isometrically by Fact~\ref{facts}.e) whereas $\mathcal F(M)_{\U}$ is not separable. Thus, $\mathcal F(M_{\U})$ is not isomorphic to $\mathcal F(M)_{\U}$. 
\end{example}

%Now we notice that the class of Lipschitz-free spaces is not stable under ultraproducts, that is there exist metric spaces $M$ such that $(\mathcal F(M))_\U$ is not isomorphic to any $\mathcal F(N)$. 

\iffalse
\begin{example}
If $\U$ is a free ultrafilter on $\mathbb N$ then $(\mathcal F(\mathbb N))_\U$ and $(\mathcal F([0,1]))_\U$ are not isomorphic to a subspace of a Lipschitz-free space.

\begin{proof} It is known  (see \cite{Henson}) that  $$(\mathcal F(\mathbb N))_\U=(\ell_1)_\U=(\ell_1(\mathbb N_\U)\oplus (L_1[0,1])_\U)_1.$$ Note that $(L_2[0,1])_\U$ is a Hilbert space and then its unit ball is a weakly compact set. Let $i\colon (L_2[0,1])_\U\to(L_1[0,1])_\U$ be the canonical inclusion. Since $(L_1[0,1])_\U=\overline{\operatorname{span}}(i(B_{(L_2[0,1])_\U}))$ and $(L_1[0,1])_\U$ is not separable, we deduce that $i(B_{(L_2[0,1])_\U})$ is a non-separable weakly compact subset of $(L_1[0,1])_\U$. So $(\mathcal F(\mathbb N))_\U$ also contains a non-separable weakly compact set. In particular, $(\mathcal F(\mathbb N))_\U$ cannot be isomorphic to a subspace of a Lipschitz-free space \textcolor{orange}{\cite{Kalton11}}. Since $(\mathcal F([0,1]))_\U=(L_1[0,1])_\U$, the same argument works for $(\mathcal F([0,1]))_\U$.\end{proof}
\end{example}\fi

In the first version of this paper, we provided some examples of metric spaces (as $M=[0,1]$ and $M=\mathbb N$) such that $\mathcal F(M)_{\U}$ is not isomorphic to a Lipschitz-free space, and we asked whether an analogous statement holds for every metric space. T. Kania has kindly provided an answer for a general metric space by strengthening our previous result.

\begin{proposition}\label{ultra-free}
Let $\mathcal U$ be a CI ultrafilter on an infinite set $I$, $M$ be a metric space and $X$ be an infinite-dimensional Banach space. Then $X_\U$ is not isomorphic to a subspace of $\mathcal F(M)$.
\end{proposition}

\begin{proof}
Since $\mathcal U$ is CI, there exists a strictly decreasing sequence $(I_n)_n$ such that $\bigcap_{n\in\mathbb N}I_n=\emptyset$. Define $(a_i)_{i\in I}$ by $a_i=\frac{1}{n}$ and $n_i=n$ if $i\in I_n\setminus I_{n+1}$. By Dvoretzky's theorem, for all $i\in I$ there exists a subspace $X_i$ of $X$ and an isomorphism $T_i\colon\ell_2^{n_i}\to X_i$ such that $$\|x\|\leq\|T_i(x)\|\leq(1+a_i)\|x\|$$ for all $x\in\ell_2^{n_i}$. Now we define $T\colon(\ell_2^{n_i})_\mathcal U\to X_\mathcal U$ by $T((x_i)_\mathcal U)=(T_i(x))_\mathcal U$. Since $\lim_\mathcal U a_i=0$, the previous inequality implies that $T$ is an isometry. We have that the ultrapower of Hilbert spaces $(\ell_2^{n_i})_\mathcal U$ is also a Hilbert space and it is non-separable (see Theorem 3.1 in \cite{Boyd}). The conclusion follows from the fact that a Lipschitz-free space does not contain a non-separable weakly compact set \textcolor{red}{\cite{Kalton11}}.
\end{proof}

\begin{corollary} Let $\mathcal U$ be a CI ultrafilter and $M$ be a infinite metric space. Then $\mathcal F(M)_{\U}$ is not isomorphic to a subspace of a Lipschitz-free space.
\end{corollary}

Thanks to Gelfand-Naimark theorem, the ultrapower of $\mathcal C(K)$-spaces is still a $\mathcal C(K)$-space, i.e.~if $\mathcal U$ is an ultrafilter on a set $I$ and if $(K_i)_{i\in I}$ is a family of compact spaces, then there exists a compact space $K$ such that $(\mathcal C(K_i))_\U=\mathcal C(K)$ isometrically. Moreover, if there is an algebra isomorphism between $(\mathcal C(K_i))_\U$ and $\mathcal C(K)$ then $(K_i)_\U$ is homeomorphic to a dense subset of $K$ \cite{Heinrich}. The following result is the analogue for $\Lip_0(K)$.

\begin{proposition}
Let $K$ be a compact metric space. Let $\mathcal U$ be an ultrafilter on a set $I$ and let $(M_i)_{i\in I}$ be a family of uniformly bounded metric spaces. If there exists an algebra isomorphism between $(\Lip_0(M_i))_\U$ and $\Lip_0(K)$, then $(M_i)_\U$ is biLipschitz equivalent to a subset of $K$.
\end{proposition}

\begin{proof}
Let $R\colon\Lip_0(K)\to(\Lip_0(M_i))_\U$ be an algebra isomorphism. If $(x_i)_\U\in(M_i)_\U$, we can define a functional $F_{(x_i)_\U}\in\Lip_0(K)^*$ by $$F_{(x_i)_\U}(f)=\lim_\U f_i(x_i)$$ for all $f\in\Lip_0(K)$ where $(f_i)_\U=R(f)$. In other words, we have $F_{(x_i)_\U}(f)=TR(f)((x_i)_\U)$ where $T\colon (\Lip_0(M_i))_\U\to (\Lip_0(M_i)_\U)$ is the operator defined in Theorem \ref{th:ultrLip}. It is clear that $F_{(x_i)_\U}$ is also multiplicative. By Lemma 7.28 in \cite{Weaver}, $F_{(x_i)_\U}$ is an evaluation, that is there exists a unique $h((x_i)_\U)\in K$ such that $F_{(x_i)_\U}=\delta_{h((x_i)_\U)}$. This allows to define a map $h\colon (M_i)_\U\to K$, we will show this is the biLipschitz map we are looking for.

Let $(x_i)_\U,(y_i)_\U\in (M_i)_\U$. We have that
\begin{align*}
d(h((x_i)_\U),h((y_i)_\U))&=\|\delta_{h((x_i)_\U)}-\delta_{h((y_i)_\U)}\| \\
&=\|F_{(x_i)_\U}-F_{(y_i)_\U}\| \\
&=\sup_{f\in B_{\Lip_0(K)}}|F_{(x_i)_\U}(f)-F_{(y_i)_\U}(f)| \\
&=\sup_{f\in B_{\Lip_0(K)}}|TR(f)((x_i)_\U)-TR(f)((y_i)_\U)|
\end{align*}
It follows that on the one hand:
$$d(h(x_i)_\U,h(y_i)_\U)\leq \sup_{f\in B_{\Lip_0(K)}} \norm{TR(f)}d((x_i)_\U, (y_i)_\U)\leq \norm{R}d((x_i)_\U, (y_i)_\U).$$
On the other hand, taking $\varepsilon>0$, there exists $(f_i)_\U\in\Lip_0(M_i)_\U$ such that $\|\delta_{(x_i)_\U}-\delta_{(y_i)_\U}\|=\langle T((f_i)_\U),\delta_{(x_i)_\U}-\delta_{(y_i)_\U}\rangle$ and $\|(f_i)_\U\|\leq 1+\varepsilon$ by Theorem \ref{th:ultrLip}. Let $g\in\Lip_0(K)$ such that $R(g)=(f_i)_\U$ and note that $\|g\|\leq(1+\varepsilon)\|R^{-1}\|$. It follows that
\begin{align*}
d((x_i)_\U,(y_i)_\U)&=\|\delta_{(x_i)_\U}-\delta_{(y_i)_\U}\|\\
&=\langle TR(g),\delta_{(x_i)_\U}-\delta_{(y_i)_\U}\rangle\\
&=(1+\varepsilon)\|R^{-1}\|\left\langle TR\left(\frac{g}{(1+\varepsilon)\|R^{-1}\|}\right),\delta_{(x_i)_\U}-\delta_{(y_i)_\U}\right\rangle\\
&\leq (1+\varepsilon)\|R^{-1}\|\sup_{f\in B_{\Lip_0(K)}}|\<TR(f),\delta_{(x_i)_\U}-\delta_{(y_i)_\U}\>|\\
&=(1+\varepsilon)\|R^{-1}\|\sup_{f\in B_{\Lip_0(K)}}|TR(f)((x_i)_\U)-TR(f)((y_i)_\U)|\\
&=(1+\varepsilon)\|R^{-1}\|d(h(x_i)_\U,h(y_i)_\U),
\end{align*}
and since $\varepsilon$ was arbitrary, we obtain that $d((x_i)_\U,(y_i)_\U)\leq \|R^{-1}\|d(h(x_i)_\U,h(y_i)_\U)$. Then we deduce that $h$ is biLipschitz. 
\end{proof}

We finish the paper remarking that the analogy with the case of ultraproducts $C(K)$-spaces is not complete. Indeed, the map $h$ constructed in the proof above does not have dense range, in general. For instance, assume $M_i=M$ is a compact metric space. Then we have $T\circ R(f)= f\circ h$ for each $f\in \Lip_0(K)$, that is, $T\circ R$ is the composition operator $C_h\colon \Lip_0(K)\to \Lip_0(M)$. Since $R$ is an isomorphism and $T$ is not injective (by Proposition \ref{prop:unifdisc}) we get that $C_h$ is not injective. It follows (see Proposition 2.25 in \cite{Weaver}) that $h(M)=\overline{h(M)}$ is properly contained in $K$.

\subsection*{Acknowledgements} The authors are very grateful to M.~Raja for suggesting this research topic and fruitful conversations, and to A.~Avil\'es and A.~Proch\'azka for their useful comments. We are also very grateful to T.~Kania for allowing us to include his proof of Proposition~\ref{ultra-free}.

The research of L.~Garc\'ia-Lirola was supported by the grants MTM2017-83262-C2-2-P and Fundaci\'on S\'eneca Regi\'on de Murcia 20906/PI/18.

The research of G.~Grelier was supported by the Grants of Ministerio de Econom\'ia, Industria y Competitividad MTM2017-83262-C2-2-P; Fundaci\'on S\'eneca Regi\'on de Murcia 20906/PI/18; and by MICINN 2018 FPI fellowship with reference PRE2018-083703, associated to grant MTM2017-83262-C2-2-P.

\end{document}